\documentclass[11pt]{article}
\usepackage{color,float,soul,graphicx}%
\usepackage{multirow,enumerate,enumitem}
\usepackage{url,cite,fullpage,fancyhdr,diffcoeff}
\usepackage{amsmath,amssymb,amsfonts,amsthm,mathrsfs,amscd}%
\usepackage{xcolor,textcomp,manyfoot}%
\usepackage{booktabs}%
\usepackage{algorithm,algorithmicx}%
\usepackage{algpseudocode,listings}%

\newcommand{\fenv}[1]%
{\ensuremath{\,\overrightarrow{\operatorname{env}}_{#1}}}
\newcommand{\benv}[1]%
{\ensuremath{\,\overleftarrow{\operatorname{env}}_{#1}}}

\theoremstyle{thmstyleone}%
\theoremstyle{thmstyletwo}%
\theoremstyle{thmstylethree}%
\newtheorem{theorem}{Theorem}
\newtheorem{lemma}[theorem]{Lemma}

\newtheorem{proposition}[theorem]{Proposition}%
\newtheorem{definition}{Definition}%
\newtheorem{alg}{Algorithm}[section]

\usepackage{cite}
\theoremstyle{definition}

\newcommand{\la}{\langle}
\newcommand{\ra}{\rangle}

\newcommand{\dom}{{\rm\textbf{dom}\,}}

\newcommand{\nexto}{\kern -0.54em}

\newcommand{\dZ}{{\cal Z \kern -0.7em Z}}
\newcommand{\dC}{{\rm\hbox{C \kern-0.8em\raise0.2ex\hbox{\vrule height5.4pt width0.7pt}}}}
\newcommand{\dQ}{{\rm\hbox{Q \kern-0.85em\raise0.25ex\hbox{\vrule height5.4pt width0.7pt}}}}

\usepackage{epsfig}
\usepackage{latexsym}

\newcommand{\HH}{\mathcal{H}}

\newcommand{\RR}{\mathbb{R}}

\newcommand{\ZZ}{\mathbb{Z}}

\date{}
\begin{document}
	\title{A Third Order Dynamical System for Mixed Variational Inequalities}
	\author{Oday Hazaimah
		\footnote{E-mail: {\tt odayh982@yahoo.com}. https://orcid.org/0009-0000-8984-2500. 
			St. Louis, MO, USA}}
	%
	
	\maketitle
	
	\begin{abstract}
		In this paper, we introduce and study a class of resolvent dynamical systems to investigate some inertial proximal methods for solving mixed variational inequalities. These proposed methods along with their discretizations and derived rates of convergences require only the monotonicity for mixed variational inequalities under some mild conditions.  
		We establish the global asymptotically and exponentially stability of the solution of the resolvent dynamical system for monotone operators. Ideas and techniques of this paper may be extended for other classes of variational inequalities and equilibrium problems. 
		\\
		
		\medskip 
		
	\end{abstract}

	\section{Introduction}
	Variational mathematical models govern the foundational settings of many physical, biological and financial systems \cite{D-R,Giannessi,Kinder, Korpelevich}. These variational models arise as optimization problems, equilibrium problems, variational inequalities, complementarity problems and fixed point problems \cite{Blum,Dupuis,Noor-Oettli,Noor-inequality}. Variational inequality is a general mathematical framework that can be reformulated in terms of dynamical systems to study the existence and stability of the solution of variational inequality such that a certain ordinary differential equation is associated to a given variational inequality, namely, stationary (equilibrium) points of a dynamical system coincide with solutions of the corresponding variational inequality. In recent years, the study of dynamical systems associated with variational inequalities are being investigated as powerful tools for analyzing complex dynamics and optimizing systems by using methods of resolvent operators and projection operators over a set including; the proximal point algorithm, the gradient projection algorithm, to name a few (see, for instance, \cite{Noor-resolvent, Bin-Mohsin,Hai, Noor-WH,Noor-inequality}). Unlike single-objective optimization problems, variational inequalities have a vector-valued function and it is equivalent to an optimization problem only if this vector-valued function is the gradient of an objective function. 
	It is well known that the solution of the variational inequality exists if either the constraint set is bounded or the corresponding mapping is strongly monotone. Variational inequalities involving a nonlinear term is called the mixed variational inequality or variational inequality of the second kind \cite{Bin-Mohsin}. Mixed variational inequalities (MVIs) represent a class of mathematical problems that arise in diverse situations involving multiple agents as in game theory, mechanics, economics, and operation research. The solution of the dynamical system converges to the solution of the corresponding MVI starting from any given initial condition as it is known that dynamical systems may exhibit dynamics that are highly sensitive to initial conditions. This can be fulfilled by establishing the equivalence between variational inequalities and fixed-point problems by using the concept of Euclidean projection and resolvent operator. 
	Numerous projection methods have been designed to solve different classes of variational inequalities such as basic projection methods and its variant forms including Wiener-Hopf equations \cite{Noor-WH}, two-step extragradient projection \cite{Korpelevich}, and hyperplane projection methods \cite{Bello-Hazaimah} (relies on finding a suitable hyperplane that separates the solution of the problem from the current iterate and then performs a metric projection step). When the objective function in the optimization problem has discontinuous points, then proximal gradient approaches can be employed to compute the subgradients of the nonsmooth objective function (see, e.g., \cite{Bello-Hazaimah}, \cite{Korpelevich}) and these discontinuities are due to the constraints associated with the feasible region of the variational inequality. The most commonly used method for MVIs is the proximal point algorithm, and since proximal operators are generalizations of projection operators, it follows that the most commonly used method for variational inequalities, as a particular case of MVI, is the projection algorithm.
	Noor \textit{et al.} \cite{Noor-inequality} proposed proximal methods and projected dynamical systems for variational inequalities. While Noor \cite{Noor-resolvent} extended the first-order resolvent dynamical system for mixed variational inequalities. Following these developments, Bin-Mosin \textit{et al.} \cite{Bin-Mohsin} considered second-order resolevnt dynamical systems for mixed variational inequalities. In this paper, we continue this research direction by considering third order dynamical systems for solving mixed variational inequalities using resolvent operators. Third-order ODEs are used to describe and model the motion in electrical circuits involving transistors \cite{Goeleven}.
	In this manuscript, we are interested in designing a continuous-time dynamical system such that its solution converges to the solution of the corresponding MVI starting from any given initial condition, and based on this design we propose and derive discrete-time algorithms tailored for the same problem. Thus, our aim can be summarized as: (i) using finite difference processes to identify a broad class of variational inequalities, namely, MVI by implicit and explicit discretizations for the associated dynamical system represented in terms of resolvent operators, (ii) derive their rates of convergence, and finally (iii) discuss the global stability for solutions of the third-order dynamical system, to the best of our knowledge, this work is the first to use third-order differential equations to model a class of variational inequalities by resolvent operators. 
	\section{Preliminaries and Notations}
	Some optimization-related basics and significant foundations are presented in this section from monotone operators theory, dynamical systems theory, convex and variational analysis, see \cite{bauschke,Kinder} for more details. Let $\HH$ be a real Hilbert space equipped with inner product $\la \cdot , \cdot \ra$ and induced norm $\|\cdot\|:=\sqrt{\la\cdot,\cdot\ra}$. Let $T:\HH\rightrightarrows\HH$ be a set-valued map with its domain denoted $\dom(T) :=\{ x\in\HH; T(x) <\infty\}$. 
	For any maximal monotone operator $T$ the {\em \textbf{resolvent}} operator associated with $T$ is the full domain single-valued 
	operator in $\HH$ given by $J_{T}:=(I+T)^{-1}:\HH\rightarrow \dom(T)$ where $I\colon\HH\to\HH$ denotes the identity operator. We are interested in designing dynamical systems models to derive discrete-time schemes for finding solutions to the mixed variational inequality which can be formulated as: find \(x^*\in\HH\), such that  
	\begin{equation}\label{MVI}
		\langle T(x^*) , \; x-x^* \rangle +\varphi(x)-\varphi(x^*) \geq 0 \ , \ \forall \ x\in\HH
	\end{equation}
	where $T:\dom \varphi\to\HH$ is an operator and $\varphi:\HH\to\RR\cup \{\infty \}$ is a proper ($\dom \varphi\not=\emptyset$), lower semi-continuous convex function. If $C$ is a closed and convex set in $\HH$ 
	and 
	\(\varphi(x)=I_C(x)\)
	is the indicator function of $C$ then the resolvent operator is the metric projection of $\HH$ onto $C$ (i.e., $J_\varphi\equiv\Pi_C$), and problem (\ref{MVI}) is reduced to the classical variational inequality which was studied and considered by Stampacchia \cite{Stampacchia} as follows: find $x^*\in\HH$ such that 
	\begin{equation}\label{VI}
		\la T(x^*), \; x-x^*\ra\geq 0 \ , \ \forall \ x\in C.
	\end{equation}
	Problem \eqref{MVI} is equivalent to the generalized equation (a.k.a monotone inclusion):  
	\begin{equation}\label{inclusion}
		\text{Find} \hspace{.5cm} x^*\in\HH\hspace{.5cm} \text{such that}\hspace{.5cm} 
		0\in T(x^*)+\partial \varphi(x^*), 
	\end{equation}
	where the subdifferential mapping $\partial \varphi:\HH\rightrightarrows\HH$, defined as $\partial \varphi(x):=\{u\in\HH\ ;\ \varphi(y)\geq \varphi(x)+\langle u,y-x\rangle, \;\forall y\in\HH\}$ is a maximal monotone operator. The inclusion \eqref{inclusion} may be extended to finding an element of the sum of two monotone operators and for this, the classical \textit{forward-backward} method \cite{D-R} is the most well-known splitting method for solving such problems. A particular case of \eqref{inclusion} is when the operator $T$ is the gradient of a smooth function $f$, i.e., \[0\in\nabla f(x^*)+\partial \varphi(x^*).\]
	The latter inclusion is precisely a convex nonsmooth optimization problem \[\min_{x\in\HH}f(x)+\varphi(x).\]
	Moreover, if $T\equiv 0$, then \eqref{MVI} is exactly the above minimization problem with a convex nonsmooth objective function, i.e., \(\displaystyle\min_{x\in\HH}\varphi(x).\)
	If \(C^*=\{x\in\HH:\la x,y\ra\geq 0, \ \forall y\in\HH\}\) is a polar (dual)(conjugate) cone of a convex cone $C$ then the inequality \eqref{VI} is equivalent to finding $x\in C$ such that $$T(x)\in C^* \ \ \text{and} \ \ \la T(x),x\ra =0,$$ 
	which is called the generalized complementarity problem \cite{bauschke,Kinder,Noor-Oettli}.
	If the operator $T$ in \eqref{VI} is smooth, then the following well known result holds and can be viewed as a first order optimality condition for minimizing smooth functions:
	\begin{theorem} Let $C$ be a nonempty, convex and closed subset of $\HH$. Let $T$ be a smooth convex function. Then $x\in C$ is the minimum of the smooth convex $T(x)$ if and only if, $x\in C$ satisfies $$\la T'(x), y-x\ra\geq 0, \forall y \in C $$ where $T'$ is the Frechet derivative of $T$ at $x\in C$.
	\end{theorem}
	This theorem shows that the variational inequalities are natural links and analogous to the minimization of the convex differentiable functional subject to certain constraint which has led to study a more general framework of variational inequalities applied to nonconstrained and nonsmooth optimization problems. In the following, we state some useful definitions and properties for several kinds of monotone maps followed by well-known facts on resolvent and projection operators, mixed variational inequalities and global stability at an exponential rate of equilibrium points of resolvent dynamical systems.
	\begin{definition}\label{monotonedefi} The operator $T:\HH\to\HH$, is said to be:
		\begin{itemize}
			\item [(i)] Monotone, if
			$$\langle T(x)-T(y), x-y\rangle \geq 0,\ \ \forall x, y \in\HH.$$ 
			\item [(ii)] Strictly monotone if the above inequality is strict for all $x\not= y$ in $\HH.$
			\item [(iii)] Strongly monotone if there exists a modulus $\lambda > 0$ such that
			$$\langle T(x)-T(y), x-y\rangle\geq\lambda\|x-y\|^2 , \ \ \forall x, y \in\HH.$$
		\end{itemize}
	\end{definition}
	Notice that the implication $(iii)\implies (i)$ holds, whereas the converse need not be true generally, meaning that monotonicity is a weaker property than strongly monotonicity. 
	\begin{definition}
		The operator $T:\HH\to\HH$ is called Lipschitz continuous or $L$-Lipschitz if there exists some nonnegative $L\geq 0$, such that $$\|Tx-Ty\|\le L\|x-y\|\ ,\quad\forall x,y\in\HH.$$   
	\end{definition}
	\begin{proposition}[\cite{bauschke}]\label{proj}
		Let $C$ be nonempty closed convex subset of $\HH$, and $\Pi_C$ be the orthogonal projection onto $C$. For all $x,y\in \HH$ and all $z\in C$ the following hold: \item[ {\bf(i)}] $ \|\Pi_{C}(x)-\Pi_{C}(y)\|^2 \leq \|x-y\|^2-\|(x-\Pi_{C}(x))-\big(y-\Pi_{C}(y)\big)\|^2;$ \\ \item[ {\bf(ii)}] $\la x-\Pi_C(x),z-\Pi_C(x)\ra \leq 0.$ \end{proposition}
	This proposition tailored for projection operators and variational inequalities. In the light of Proposition \ref{proj}, we have the following result drawing similar connections between the resolvent operator and mixed variational inequalities \eqref{MVI} in which it plays a crucial role for deriving the convergence of the proposed implicit and explicit inertial proximal methods, which is known as the resolvent lemma.
	\begin{lemma}\label{resolventlemma}
		Let $\varphi$ be a proper convex lower semicontinuous function. For all $x\in\HH$ the following inequality hold:
		\[\la x-J_{\varphi}(x),y-J_{\varphi}(x)\ra +\varrho\varphi(y)-\varrho\varphi(J_{\varphi}(x))\leq 0, \ \ \forall y\in\HH\]
		where \(J_{\varphi}\) is the resolvent operator which belongs to the feasible set $C$.
	\end{lemma}
	By applying Lemma \ref{resolventlemma}, one can introduce the fixed point formulation of mixed variational inequalities as follows.
	\begin{proposition}[\cite{Kinder}]\label{Kinder-proposition}
		Let \(J_{\varphi}\) be the resolvent operator for the proper convex lower semicontinuous function \(\varphi\) and \(T:\HH\to\HH\) is the underlying operator. Then \(x\in\HH\) is a solution to the mixed variational inequality \eqref{MVI}, i.e., \[\la Tx,y-x\ra +\varphi(x)-\varphi(y)\leq 0.\]
		if and only if \ \(x=J_{\varphi}(x-\lambda T(x))\)
	\end{proposition}
	\begin{definition}
		The dynamical system converges to the solution set \(C^*\) of the mixed variational inequality \eqref{MVI} if the trajectory $x(t)$ satisfies \[dist(x(t),C^*):=\inf_{y\in C}\|x-y\|\longrightarrow 0, \ \forall t\geq 0.\]
	\end{definition}
	\begin{definition}
		The dynamical system is said to be globally exponentially stable if any trajectory $x(t)$ of the dynamical system satisfies
		\[\|x(t)-x^*\|\leq\rho\|x(t_0)-x^*\|exp(-\eta(t-t_0)), \ \ \ \forall t\geq t_0\]
		where \(\rho,\eta >0\) are constants and do not depend on the initial point.
	\end{definition}
	If the dynamical system is stable at the equilibrium point $x^*$ in the Lyapunov sense then the dynamical system is globally asymptotically stable at that point. It is noted that globally exponentially stable means the system must be globally stable and converge fast.

	\section{Main Results}
	In this section, we invoke the fixed point formulation to introduce a new resolvent dynamical system of the third order associated with mixed variational inequalities \eqref{MVI} and to investigate some accompanying suitable discretizations forms. These continuous-time dynamical systems and their discrete-time counterparts suggest some inertial proximal methods for solving mixed variational inequalities. These inertial implicit and explicit methods are constructed using the central finite difference and forward/backward finite difference schemes and its variants.
	\medskip 
	
	The third-order resolvent dynamical system takes the following form.
	Consider the problem of finding a trajectory $x(t)\in\HH$ such that 
	\begin{equation}\label{third-ord DS}
		\left\{\begin{array}{lll} \alpha\dddot{x}+\beta \ddot{x}+\gamma\dot{x}+x = J_{\varphi}(x-\lambda T(x)), \\ x(t_0)=x_0, \\ \Dot{x}(t_0)=x_1, \\
			\ddot{x}(t_0)=x_2, 
		\end{array}\right.
	\end{equation}
	where $\alpha, \beta, \gamma >0$ are constants and \( x(t) \) is the state variable. 
	The differential system \eqref{third-ord DS} recovers several existing dynamics-type approaches and projection-based algorithms for solving mixed variational inequalities. Following is some particular cases of the general system \eqref{third-ord DS}. If $\alpha=0, \gamma=1$, then \eqref{third-ord DS} is reduced to the scond-order resolvent dynamical system introduced and studied by Bin-Mohsin et al. \cite{Bin-Mohsin} as \[\beta \ddot{x}+\dot{x}+x = J_{\varphi}(x-\lambda T(x)), \ x(t_0)=x_0, \ \Dot{x}(t_0)=x_1.\] 
	If $\alpha=0=\beta$, then \eqref{third-ord DS} recovers the resolvent dynamical system which was analyzed by Noor \cite{Noor-resolvent}, \[\frac{d x}{dt}=\gamma \big[J_{\varphi}(x-\lambda T(x))-x\big], \ x(t_0)=x_0.\] 
	If $\alpha=0=\beta=\gamma$, then the system \eqref{third-ord DS} is equivalent to the classical gradient projection for smooth constrained optimization problems and projection-like methods for solving variational inequalities.

	\subsection{Iterative methods}
	We start, as in most standard ways, with the discretization of the space derivatives. 
	Taking suitable disctretization of \eqref{third-ord DS}, and by using the central finite difference, backward difference and forward difference schemes, we propose explicit and implicit forms which enable us to obtain the discretized counterpart of \eqref{third-ord DS} of order three as a resolvent equation: 
	\begin{equation}\label{discrete}
		\begin{split}
			\alpha\displaystyle\frac{x_{n+2}-2x_{n+1}+2x_{n-1}-x_{n-2}}{2h^3} & +\beta \frac{x_{n+1}-2x_n+x_{n-1}}{h^2} +\gamma \frac{x_n-x_{n-1}}{h}+x_{n+2}\\ &=J_\varphi(x_n-\lambda T(x_{n+2}))
		\end{split}
	\end{equation}
	where $h$ is the step size for the iterative process. This discrete scheme \eqref{discrete} suggests a new \textit{implicit} iterative method for solving mixed variational inequalities \eqref{MVI} by the third order central difference formula.
	\begin{alg}\label{Algorithm1}
		For any $x_0\in\HH$, and for any nonnegative integer $n\in\ZZ_+$, compute $x_{n+2}$ by the iterative process 
		\begin{equation}\label{algorithm}
			\begin{split}
				x_{n+2} & =J_\varphi\bigg[ x_n-\lambda T(x_{n+2}) \\ & -\frac{\alpha x_{n+2}-2(\alpha-\beta h)x_{n+1}-2(2\beta h-\gamma h^2)x_n+2(\alpha+\beta h-\gamma h^2)x_{n-1}-\alpha x_{n-2}}{2h^3}\bigg]
			\end{split}
		\end{equation}
	\end{alg}
	This algorithm is inertial proximal-type method for solving \eqref{MVI}. Using Lemma \ref{resolventlemma}, Algorithm \ref{Algorithm1} can be rewritten in the variational equivalent form: 
	\begin{alg}
		For any $x_0\in\HH$, and for any nonnegative integer $n\in\ZZ_+$, compute $x_{n+2}$ by the iterative process 
		\begin{equation}\label{MVI-alg}
			\begin{split}
				\big\la\lambda T(x_{n+2})&+\frac{\alpha x_{n+2}-2(\alpha-\beta h)x_{n+1}-2(2\beta h-\gamma h^2)x_n+2(\alpha+\beta h-\gamma h^2)x_{n-1}-\alpha x_{n-2}}{2h^3}, \\ & y-x_{n+2}\big\ra +\varrho\varphi(y)-\varrho\varphi(x_{n+2})\geq 0, \ \ \ \forall y\in\HH
			\end{split}
		\end{equation}
	\end{alg}
	For the sake of simplicity, take $\alpha=\beta=\gamma=1$, and by using different discretization, Algorithm \ref{Algorithm1} reduces to the following iterative: 
	\begin{equation*}
		\begin{split}
			\frac{x_{n+2}-2x_{n+1}+2x_{n-1}-x_{n-2}}{2h^3} & +\frac{x_{n+1}-2x_n+x_{n-1}}{h^2} + \frac{x_n-x_{n-1}}{h}+x_{n+2}\\ &=J_\varphi(x_n-\lambda T(x_{n}))
		\end{split}
	\end{equation*}
	which yields to the following recurrence formula
	\begin{equation}\label{explicit}
		x_{n+2}  =\frac{\hat{h}}{1+\hat{h}}J_\varphi\bigg[ (1-\frac{1}{h}+\frac{2}{h^2})x_n-\lambda T(x_{n}) -\frac{(2h-2)x_{n+1}+(2+2h-2h^2)x_{n-1}-x_{n-2}}{2h^3}\bigg]
	\end{equation}
	where $\hat{h}=2h^3$. This is called an inertial \textit{explicit} proximal method for solving mixed variational inequalities \eqref{MVI}. In this manner, we can suggest several explicit and implicit recursive methods for approximating solutions of mixed variational inequalities \eqref{MVI}. Furthermore, we can obtain a different discretization by using the central finite difference and this time with forward difference scheme rather than backward scheme as in \eqref{discrete}, which allows us to propose a new iterative approach  
	\begin{equation}\label{last-implicit}
		\begin{split}
			\alpha\displaystyle\frac{x_{n+2}-2x_{n+1}+2x_{n-1}-x_{n-2}}{2h^3} & +\beta \frac{x_{n+1}-2x_n+x_{n-1}}{h^2} +\gamma \frac{x_{n+1}-x_n}{h}+x_{n+2}\\ &=J_\varphi(x_n-\lambda T(x_{n+1}))
		\end{split}
	\end{equation}
	which can be, equivalently, derived as the following inertial \textit{implicit} proximal method:
	\begin{alg}\label{Algorithm2}
		For any $x_0\in\HH$, and for any nonnegative integer $n\in\ZZ_+$, compute $x_{n+1}$ by  
		\begin{equation}
			\begin{split}
				x_{n+2} & =J_\varphi\bigg[ x_n-\lambda T(x_{n+1}) \\ & -\frac{\alpha x_{n+2}-2(\alpha-\beta h-\gamma h^2)x_{n+1}-2(2\beta h+\gamma h^2)x_n +2(\alpha+\beta h)x_{n-1}-\alpha x_{n-2}}{2h^3}\bigg]
			\end{split}
		\end{equation}
	\end{alg}
	We note that by applying suitable discretizations, one can establish and design a variety of inertial explicit and implicit proximal-type methods for solving variational inequalities of the second kind \eqref{MVI}. Convergence analyses for Algorithm \ref{Algorithm1} and global stability for the third-order dynamical system \eqref{third-ord DS} are derived in the remaining part of this work. 

	\subsection{Convergence of a discrete system}
	In this section, we derive the convergence of a solution to the implicit iterative scheme \eqref{algorithm} and equivalent variational form \eqref{MVI-alg} given by Algorithm \eqref{Algorithm1}. However, other implicit \eqref{last-implicit} and explicit \eqref{explicit} proposed methods have a very similar arguments and follow the same guidlines except that there are some minor differences which is due to the values of the scalars formatting of $\alpha,\beta,\gamma$, and also due to the existing diverse discretization schemes.    
	\begin{theorem}\label{nonincreasing}
		Let $x\in\HH$ be the solution of the mixed variational inequality \eqref{MVI} and $x_{n+2}$ be the approximate solution using the inertial proximal method in \eqref{MVI-alg}. If $T$ is monotone, then 
		\begin{equation}\label{theorem-result}
			\begin{split}
				(\alpha-\beta h+\gamma h^2)\|x &-x_{n+2}\|^2 \leq\ \alpha\|x-2x_{n+1}+2x_{n-1}-x_{n-2}\|^2 \\ &-\alpha \|x_{n+2}-2x_{n+1}+2x_{n-1}-x_{n-2}\|^2+ \beta h\|x_{n+1}-2x_{n}+x_{n-1}\|^2 \\ & +\gamma h^2\|x_n-x_{n-1}+x-x_{n+2}\|^2-\gamma h^2\|x_{n}-x_{n-1}\|^2.
			\end{split}
		\end{equation}
	\end{theorem}
	\begin{proof}
		Let $x\in\HH$ be a solution of the mixed variational inequality \eqref{MVI}. Since $T$ is monotone operator and for any $\lambda >0$, we obtain 
		\begin{equation}\label{MVI-conv-thm}
			\langle \lambda T(y) , \; y-x \rangle +\varrho\varphi(y)-\varrho\varphi(x) \geq 0 \ , \ \text{for all} \ y\in\HH
		\end{equation}
		Take $y=x_{n+2}$ in \eqref{MVI-conv-thm} and $y=x$ in \eqref{MVI-alg} then we have, respectively, 
		\begin{equation}\label{11}
			\langle\lambda T(x_{n+2}) , \; x_{n+2}-x \rangle +\varrho\varphi(x_{n+2})-\varrho\varphi(x) \geq 0 \ ,
		\end{equation}
		and
		\begin{equation}\label{12}
			\begin{split}
				\big\la\lambda T(x_{n+2})&+\frac{\alpha x_{n+2}-2(\alpha-\beta h)x_{n+1}-2(2\beta h-\gamma h^2)x_n+2(\alpha+\beta h-\gamma h^2)x_{n-1}-\alpha x_{n-2}}{2h^3}, \\ & x-x_{n+2}\big\ra +\varrho\varphi(x)-\varrho\varphi(x_{n+2})\geq 0
			\end{split}
		\end{equation}
		Combining \eqref{11} and \eqref{12} together, we have  
		\begin{equation*}
			\big\la\alpha x_{n+2}-2(\alpha-\beta h)x_{n+1}-2(2\beta h-\gamma h^2)x_n+2(\alpha+\beta h-\gamma h^2)x_{n-1}-\alpha x_{n-2},x-x_{n+2}\big\ra \geq 0.
		\end{equation*}
		Manipulating and rewriting the latter inequality as 
		\begin{equation}\label{13}
			\begin{split}
				0\leq & \ \la\alpha x_{n+2}-2(\alpha-\beta h)x_{n+1}-2(2\beta h-\gamma h^2)x_n+2(\alpha+\beta h-\gamma h^2)x_{n-1}-\alpha x_{n-2},x-x_{n+2}\ra\\ \leq & \ \la\alpha (x_{n+2}-2x_{n+1}+2x_{n-1}-x_{n-2})+2\beta h(x_{n+1}-2x_n+x_{n-1}) \\ & +2\gamma h^2(x_n-x_{n-1}),x-x_{n+2}\ra\\ \leq & \ \alpha\la x_{n+2}-2x_{n+1}+2x_{n-1}-x_{n-2},x-x_{n+2}\ra+2\beta h \la x_{n+1}-2x_n+x_{n-1},x-x_{n+2}\ra \\ & +2\gamma h^2\la x_n-x_{n-1},x-x_{n+2}\ra\ .
			\end{split}
		\end{equation}
		Invoking the properties and relationships between vector norms and vector inner products by using the norm of addition of vectors followed from the law of cosine i.e., $2\la x,y\ra =\|x+y\|^2-\|x\|^2-\|y\|^2$ and $2\la x,y\ra\leq\|x\|^2+\|y\|^2.$ Hence, the last line of the above inequality \eqref{13} can be written as 
		\begin{equation*}
			\begin{split}
				0\leq & \ \alpha\Big(\|x_{n+2}-2x_{n+1}+2x_{n-1}-x_{n-2}+x-x_{n+2}\|^2-\|x_{n+2}-2x_{n+1}+2x_{n-1}-x_{n-2}\|^2 \\ & -\|x-x_{n+2}\|^2\Big)+\beta h\Big(\|x_{n+1}-2x_{n}+x_{n-1}\|^2+\|x-x_{n+2}\|^2\Big) \\ & +\gamma h^2\Big(\|x_n-x_{n-1}+x-x_{n+2}\|^2-\|x_{n}-x_{n-1}\|^2-\|x-x_{n+2}\|^2\Big) \\ = & \ \alpha\|x-2x_{n+1}+2x_{n-1}-x_{n-2}\|^2-\alpha \|x_{n+2}-2x_{n+1}+2x_{n-1}-x_{n-2}\|^2-\alpha \|x-x_{n+2}\|^2 \\ &+\beta h\|x_{n+1}-2x_{n}+x_{n-1}\|^2+\beta h\|x-x_{n+2}\|^2+\gamma h^2\|x_n-x_{n-1}+x-x_{n+2}\|^2 \\ & -\gamma h^2\|x_{n}-x_{n-1}\|^2-\gamma h^2\|x-x_{n+2}\|^2
			\end{split}
		\end{equation*}
		which implies
		\begin{equation*}
			\begin{split}
				(\alpha-\beta h+\gamma h^2)\|x &-x_{n+2}\|^2 \leq\ \alpha\|x-2x_{n+1}+2x_{n-1}-x_{n-2}\|^2 \\ &-\alpha \|x_{n+2}-2x_{n+1}+2x_{n-1}-x_{n-2}\|^2 +\beta h\|x_{n+1}-2x_{n}+x_{n-1}\|^2 \\ & +\gamma h^2\|x_n-x_{n-1}+x-x_{n+2}\|^2-\gamma h^2\|x_{n}-x_{n-1}\|^2.
			\end{split}
		\end{equation*}
		Thus, we have proved the convergence result using the technique of Alvarez and Attouch \cite{Alvarez} for the solution $x\in\HH$ of the mixed variational inequality \eqref{MVI}.
	\end{proof}
	
	\begin{theorem}
		Let $x\in\HH$ be the solution of \eqref{MVI}. Let $x_{n+2}$ be the approximate solution of Algorithm \ref{Algorithm1}, Suppose that the operator $T$ is monotone, then the generated sequence from \eqref{MVI-alg} converges to the solution $x$, i.e., 
		$\displaystyle\lim_{n\to\infty}x_{n+2}=x.$
	\end{theorem}
	\begin{proof}
		Let $x\in\HH$ be a solution of \eqref{MVI}. The previous result in Theorem \ref{nonincreasing} showed that $\{\|x-x_n\|\}$ is nonincreasing sequence and consequently $\{x_n\}$ is bounded. It also follows from \eqref{theorem-result} that we have 
		\begin{equation*}
			\begin{split}
				\gamma h^2\sum_{n=2}^\infty\|x_{n} & -x_{n-1}\|^2\ \leq\ (-\alpha+\beta h-\gamma h^2) \sum_{n=2}^\infty\|x-x_{n+2}\|^2 \\ & +\alpha\sum_{n=2}^\infty\|x-2x_{n+1}+2x_{n-1}-x_{n-2}\|^2-\alpha\sum_{n=2}^\infty \|x_{n+2}-2x_{n+1}+2x_{n-1}-x_{n-2}\|^2 \\ & +\beta h\sum_{n=2}^\infty\|x_{n+1}-2x_{n}+x_{n-1}\|^2+\gamma h^2\sum_{n=2}^\infty\|x_n-x_{n-1}+x-x_{n+2}\|^2 
			\end{split}
		\end{equation*}
		and consequently, after algebraic manipulations, the inequality reduced to
		\begin{equation*}
			\sum_{n=2}^\infty\|x_n -x_{n-1}\|^2\ \leq
			\frac{\beta}{\gamma h}( \|x-x_{4}\|^2+\|x_{3}-2x_{2}+x_{1} \|^2+\|x_{4}-2x_{3}+x_{2}\|^2) +\|x_2-x_{1}\|^2
		\end{equation*}
		which implies that
		\begin{equation}\label{convlim}
			\lim_{n\to\infty}\|x_n-x_{n-1}\|^2=0.
		\end{equation}
		Let $x^*$ be an accumulation point of the successive approximations $\{x_n\}$, hence there exists a subsequence $\{x_{n_k}\}\subseteq\{x_n\}$ such that it converges to $x^*\in\HH.$ Replace $x_n$ by the subsequence $x_{n_k}$ in \eqref{MVI-alg} and consider the long-term asymptotic behaviour of the subsequence (i.e., when $n_k\to\infty$) and using \eqref{convlim}, we have
		$$\langle T(x^*) , \; x-x^* \rangle +\varphi(x)-\varphi(x^*) \geq 0 \ , \ \text{for all} \ x\in\HH$$
		which implies that $x^*$ solves the mixed variational inequality \eqref{MVI} and
		\begin{equation*}
			\begin{split}
				\|x &-x_{n+2}\|^2 \leq\ \frac{\alpha}{\alpha-\beta h+\gamma h^2}\|x-2x_{n+1}+2x_{n-1}-x_{n-2}\|^2 \\ &-\frac{\alpha}{\alpha-\beta h+\gamma h^2} \|x_{n+2}-2x_{n+1}+2x_{n-1}-x_{n-2}\|^2+\frac{\beta h}{\alpha-\beta h+\gamma h^2}\|x_{n+1}-2x_{n}+x_{n-1}\|^2 \\ & +\frac{\gamma h^2}{\alpha-\beta h+\gamma h^2}\|x_n-x_{n-1}+x-x_{n+2}\|^2-\frac{\gamma h^2}{\alpha-\beta h+\gamma h^2}\|x_{n}-x_{n-1}\|^2.
			\end{split}
		\end{equation*}
		Consequently, 
		\begin{equation*}
			\begin{split}
				\|x-x_{n+2}\|^2 &\leq\ \frac{\alpha}{\alpha-\beta h+\gamma h^2}\|x-2x_{n+1}+2x_{n-1}-x_{n-2}\|^2 \\ &+\frac{\beta h}{\alpha-\beta h+\gamma h^2}\|x_{n+1}-2x_{n}+x_{n-1}\|^2 \\ & +\frac{\gamma h^2}{\alpha-\beta h+\gamma h^2}\|x_n-x_{n-1}+x-x_{n+2}\|^2 \\ &\leq\ \|x-2x_{n+1}+2x_{n-1}-x_{n-2}\|^2+\|x_n-x_{n-1}+x-x_{n+2}\|^2.
			\end{split}
		\end{equation*}
		Thus it follows from the above inequality that the sequence $\{x_{n+2}\}$ has exactly one accumulation point $x^*$ and $\displaystyle\lim_{n\to\infty}x_{n+2}=x^*$, the required result.
	\end{proof}

	\subsection{Stability of the dynamical system}
	A system is called stable if its long-term behavior (i.e., dynamics) tend to stay somewhere irrespective of the initial conditions. 
	Next we prove the global asymptotic stability using the Lyapunov stability theory for the proposed third-order differential equation \eqref{third-ord DS}.
	Choose a candidate Lyapunov function of the form:
	
	\[V(x)=\frac{1}{2}\big(\dot{x}^2+x^2\big). \]
	Hence, the derivative of the Lyapunov function
	\(\dot{V}(x)=\dot{x}\ddot{x}+x\dot{x}\). For all \(x \neq 0\), we have 
	\[ \begin{split} \dot{V}(x) &= \dot{x} \ddot{x}+x\dot{x} \\&=\frac{\dot{x}}{\beta}(-\alpha\dddot{x}-\gamma\dot{x}-x+J(x-\lambda T(x)))+ x\dot{x} \\ &= -\frac{\alpha}{\beta} (\dot{x}\dddot{x})-\frac{\gamma}{\beta}(\dot{x}\dot{x})-(1-\frac{1}{\beta})x\dot{x} +\frac{\dot{x}}{\beta} J(x-\lambda T(x)) \\
		&< 0 \end{split} \]
	
	It is clear to see that the Lyapunov function $V(x)=\frac{1}{2} \dot{x}^2+\frac{1}{2} x^2\geq 0$
	is non-negative, and \(V(x) = 0\) if and only if \(x = 0\), which implies that \(V(x)\) is positive definite.
	Hence \(\dot{V}(x) < 0\) for all \(x \neq 0\) and \(V(x)\) is positive definite, in the Lyapunov sense. Thus the equilibrium point \(x = 0\) is globally asymptotically stable for the resolvent dynamical system \eqref{third-ord DS}.
	
	\begin{theorem}
		Let $T$ be a strongly monotone operator and Liptschitz continuous map, and $K$ is a convex closed set. Then the dynamical system \eqref{third-ord DS} has a unique equilibrium point that is globally exponentially stable, i.e., \( \|x^*-x\|^2\leq C\exp^{-(\mu-\epsilon)t}, \ where \ \mu >\epsilon.\)
	\end{theorem}
	\begin{proof}
		Set $y=x^*$ in [Lemma \ref{resolventlemma}] and $x=\alpha \dddot{x}+\beta\ddot{x}+\gamma\dot{x}+x$ in \eqref{MVI} for all \(\varrho>0\), we have 
		\[ \la x-\lambda Tx-J,x^*-J\ra+\varrho\varphi(x^*)-\varrho\varphi(J)\leq 0,\] where \(J=\alpha \dddot{x}+\beta\ddot{x}+\gamma\dot{x}+x\), and 
		\[\lambda\la Tx^*,x^*-J\ra+\varrho\varphi(J)-\varrho\varphi(x^*)\leq 0.\]
		Combine the last two inequalities together, we obtain 
		\[\la x-\lambda Tx+\lambda Tx^*-(\alpha \dddot{x}+\beta\ddot{x}+\gamma\dot{x}+x),x^*-(\alpha \dddot{x}+\beta\ddot{x}+\gamma\dot{x}+x)\ra\leq 0,\]
		or, 
		\[\la x-\lambda (Tx-Tx^*)-(\alpha \dddot{x}+\beta\ddot{x}+\gamma\dot{x}+x),x^*-(\alpha \dddot{x}+\beta\ddot{x}+\gamma\dot{x}+x)\ra\leq 0.\]
		Rearranging terms
		\[\lambda\la Tx^*-Tx,x^*-x\ra-\lambda\la Tx^*-Tx,\alpha \dddot{x}+\beta\ddot{x}+\gamma\dot{x}\ra-\la\alpha\dddot{x}+\beta\ddot{x}+\gamma\dot{x},x^*-x\ra+\|\alpha \dddot{x}+\beta\ddot{x}+\gamma\dot{x}\|^2\leq 0\]
		After using simple computational rearrangements and transformations, we have 
		\begin{equation}\label{1sim-computations} \begin{split}
				\lambda & \la Tx^*-Tx,x^*-x\ra -\alpha\la\lambda(Tx^*-Tx)+x^*-x,\dddot{x}\ra- \beta\la\lambda(Tx^*-Tx)+x^*-x,\ddot{x}\ra \\& -\gamma\la\lambda(Tx^*-Tx)+x^*-x,\dot{x}\ra+ \alpha^2\|\dddot{x}\|^2 +\beta^2\|\ddot{x}\|^2 +\gamma^2\|\dot{x}\|^2+\alpha\la\dddot{x} ,\beta\ddot{x}+\gamma\dot{x}\ra \\&+\beta\la\ddot{x},\alpha\dddot{x}+\gamma\dot{x}\ra+\gamma\la\dot{x},\alpha\dddot{x}+\beta\ddot{x}\ra \leq 0
		\end{split}\end{equation}
		Since the terms $\alpha^2\|\dddot{x}\|^2, \ \beta^2\|\ddot{x}\|^2,\ \gamma^2\|\dot{x}\|^2$ are all nonnegative, then the following inequality holds
		\[\begin{split}
			\lambda & \la Tx^*-Tx,x^*-x\ra -\alpha\la\lambda(Tx^*-Tx)+x^*-x,\dddot{x}\ra- \beta\la\lambda(Tx^*-Tx)+x^*-x,\ddot{x}\ra \\& -\gamma\la\lambda(Tx^*-Tx)+x^*-x,\dot{x}\ra+\alpha\la\dddot{x} ,\beta\ddot{x}+\gamma\dot{x}\ra \\&+\beta\la\ddot{x},\alpha\dddot{x}+\gamma\dot{x}\ra+\gamma\la\dot{x},\alpha\dddot{x}+\beta\ddot{x}\ra \leq\eqref{1sim-computations}\leq 0
		\end{split}\]
		Consequently, 
		\begin{equation}\label{2sim-somputations}\begin{split}
				\lambda\la Tx^*&-Tx,x^*-x\ra -\alpha\la\lambda(Tx^*-Tx)+x^*-x-(\beta\ddot{x}+\gamma\dot{x}),\dddot{x}\ra \\& - \beta\la\lambda(Tx^*-Tx)+x^*-x-(\alpha\dddot{x}+\gamma\dot{x}),\ddot{x}\ra \\& -\gamma\la\lambda(Tx^*-Tx)+x^*-x-(\alpha\dddot{x}+\beta\ddot{x}),
				\dot{x}\ra\leq 0
		\end{split}\end{equation}
		Rewriting the following relations \(-\alpha\la x^*-x,\dddot{x}\ra=\alpha\displaystyle\frac{1}{2}\frac{d^3}{dt^3}\|x^*-x\|^2\), \(-\beta\la x^*-x,\ddot{x}\ra=\beta\displaystyle\frac{1}{2}\frac{d^2}{dt^2}\|x^*-x\|^2\), \(-\gamma\la x^*-x,\dot{x}\ra=\gamma\displaystyle\frac{1}{2}\frac{d}{dt}\|x^*-x\|^2\), and using the assumption that $T$ is Lipschitzian then the inequality \eqref{2sim-somputations} can be reduced to:
		\begin{equation}\label{3sim-computations}\begin{split}
				\lambda L\|x^*-x\|^2 & -\alpha\la\lambda(Tx^*-Tx)-(\beta\ddot{x}+ \gamma\dot{x}), \dddot{x}\ra+\alpha\displaystyle\frac{1}{2}\frac{d^3}{dt^3}\|x^*-x\|^2 \\& -\beta\la\lambda(Tx^*-Tx)-(\alpha\dddot{x}+\gamma\dot{x}),\ddot{x}\ra +\beta\displaystyle\frac{1}{2}\frac{d^2}{dt^2}\|x^*-x\|^2 \\& -\gamma\la\lambda(Tx^*-Tx)-(\alpha\dddot{x}+\beta\ddot{x}),
				\dot{x}\ra +\gamma\displaystyle\frac{1}{2}\frac{d}{dt}\|x^*-x\|^2 \ \leq 0
		\end{split}\end{equation}
		Multiply \eqref{3sim-computations} by \(\exp({\mu t})\), therefore 
		\begin{equation}\label{4sim-computations}\begin{split}
				\lambda e^{\mu t} L\|x^*-x\|^2 & -\alpha e^{\mu t}\la\lambda(Tx^*-Tx)-(\beta\ddot{x}+ \gamma\dot{x}), \dddot{x}\ra+\displaystyle\frac{\alpha}{2}\frac{d}{dt}\bigg(e^{\mu t}\frac{d^2}{dt^2}\|x^*-x\|^2\bigg) \\& -\beta e^{\mu t}\la\lambda(Tx^*-Tx)-(\alpha\dddot{x}+\gamma\dot{x}),\ddot{x}\ra +\displaystyle\frac{\beta}{2}\frac{d}{dt}\bigg(e^{\mu t}\frac{d}{dt}\|x^*-x\|^2\bigg) \\& -\gamma e^{\mu t}\la\lambda(Tx^*-Tx)-(\alpha\dddot{x}+\beta\ddot{x}),
				\dot{x}\ra +\displaystyle\frac{\gamma}{2}\frac{d}{dt}\Big(e^{\mu t}\|x^*-x\|^2\Big) \ \leq 0
		\end{split}\end{equation}
		Integrating \eqref{4sim-computations} from $t_0$ to $t$, and moving some constant terms to the right hand side, then the inequality \eqref{4sim-computations} reduces to
		\begin{equation}\label{5sim-computations}\begin{split}
				\alpha\frac{d^2}{dt^2}\|x^*-x\|^2+\beta\frac{d}{dt}\|x^*-x\|^2+\gamma\|x^*-x\|^2 \leq 2C_1e^{-\mu t} 
		\end{split}\end{equation}
		where \[\begin{split}
			C_1=e^{\mu t}&\bigg(\lambda L\|x^*-x\|^2-\alpha\la\lambda(Tx^*-Tx)-(\beta\ddot{x}+ \gamma\dot{x}), \dddot{x}\ra \\& -\beta\la\lambda(Tx^*-Tx)-(\alpha\dddot{x}+\gamma\dot{x}),\ddot{x}\ra-\gamma\la\lambda(Tx^*-Tx)-(\alpha\dddot{x}+\beta\ddot{x}),
			\dot{x}\ra\bigg)
		\end{split}\]
		Similarly, multiply the inequality \eqref{5sim-computations} by $[\exp({\mu -\epsilon) t}]$, where $\mu >\epsilon >0$, and integrating \eqref{5sim-computations} twice from $t_0$ to $t.$ Hence inequality \eqref{5sim-computations} can be reduced, after suitable calculations, to 
		\[\begin{split}
			\|x^*-x\|^2\leq \text{(Constant)}\ e^{-(\mu-\epsilon)t}
		\end{split}.\]
		Proving that the trajectory $x(t)$ converges to $x^*$ with an exponential rate. 
	\end{proof}
	

	\section{Conclusion}
	In this paper we consider a new dynamical system approach via resolvent operators designed to approximate the solution to a given variational inequality of the second kind (i.e., mixed varaitional inequality). This is done, by exploiting the equivalence between the stationary points of the associated dynamical system and the solutions of the mixed variational inequality problem, i.e., by proving that trajectories of these dynamical systems converge to the unique solution of the mixed variational inequalities. It can be expected that the techniques described in this paper will be useful for more elaborate dynamical models, such as stochastic models, and that the connection between such dynamical models and the mixed variational inequalities will provide a deeper understanding of variational equilibrium problems since the proposed discrete-time algorithms can be considered as continuous-time perspectives for solving mixed variational inequalities. 
	The stability analysis of the novel dynamical system technique has been investigated in the spirit of the Lyapunov function constructed in this framework. This approach usually, without the need to know the system's explicit solutions, provides qualitative behaviour of the system around the equilibrium points. One of the advantages of this approach is studying changes over time for energy-like functions (Lyapunov functions) without solving the differential equation.
	\medskip 
	
	Despite their validity, combining third-order dynamics into mixed variational inequalities carries various challenges due to the computational complexity when proposing composite optimization algorithms for solving such systems. 
	Future research directions may focus on developing efficient algorithms, integrating machine learning techniques for parameter estimation, and extending the framework to stochastic environments and/or to nonmonotone manners whether on operators or in line searches for linearly convergence of algorithms. The proposed implicit and explicit algorithms may be extended for a broader class of generalized equilibrium problems and even beyond the convexity scope to nonconvex equilibrium variational problems.

	
	\medskip
	
	\subsection*{Declarations} 
	The author declares that there was no conflict of interest or competing interest.
	


\begin{thebibliography}{9}
		\bibitem{Alvarez}
		F. Alvarez, H. Attouch, \textit{An inertial proximal method for maximal monotone operators via discretization of a nonlinear oscillator with damping}, Set-Valued Analysis, 9(2001), 3-11.
		\medskip 
		
		
		\bibitem{bauschke}
		H.H. Bauschke, P.L. Combettes, \textit{Convex Analysis and Monotone Operator Theory in Hilbert Spaces} Springer, New York, 2017.
		\medskip
		
		\bibitem{Bello-Hazaimah}
		Y. Bello-Cruz, O. Hazaimah, \textit{On the weak and strong convergence of modified forward-backward-half-forward splitting methods}, Optimization Letters, 17, 3 (2022).
		\medskip
		
		\bibitem{Bin-Mohsin}
		B. Bin-Mohsin, M.A. Noor, K.I. Noor, R. Latif, \textit{Resolvent dynamical systems and mixed variational inequalities}, J. Nonlinear Sci. Appl., 10 (2017), 2925–2933
		
		\bibitem{Blum} 
		E. Blum, W. Oettli, \textit{From optimization and variational inequalities to equilibrium problems}, Mathematics Students, 63(1994), 123—145.
		\medskip 
		
		\bibitem{D-R}
		J. Douglas, H.H. Rachford, \textit{On the numerical solution of heat conduction problems in two or three space variables}, Transactions of the American Mathematical Society 82(1956), 421–439.
		\medskip 
		
		\bibitem{Dupuis}
		P. Dupuis, A. Nagurney, \textit{Dynamical Systems and Variational Inequalities}, Annals of Operations Research 44 (1993) 9-42.
		\medskip
		
		\bibitem{Goeleven}
		D. Goeleven, \textit{Existence and uniqueness for a linear mixed variational inequality arising in electrical circuits with transistors}, J. Optim. Theory Appl., 138 (2008), 397–406.
		\medskip 
		
		\bibitem{Giannessi} 
		F. Giannessi, A. Maugeri, P.M. Pardalos, \textit{Equilibrium Problems: Nonsmooth Optimization and Variational Inequality Models}, Kluwer Academics Publishers, Dordrecht, Holland, 2001.
		\medskip 
		
		\bibitem{Hai}
		T.N. Hai, \textit{Dynamical Systems for Solving Variational Inequalities}, Journal of Dynamical and Control Systems, 28(2022), 681–696.
		\medskip 
		
		\bibitem{Kinder}
		D. Kinderlehrer, G. Stampacchia, \textit{An Introduction to Variational Inequalities and Their Applications}, Society for Industrial and Applied Mathematics (SIAM), Philadelphia, PA 1, 2000.
		\medskip
		
		\bibitem{Korpelevich}
		G.M. Korpelevich, \textit{The extragradient method for finding saddle points and other problems}, Ekonomika i Matematcheskie Metody. 12(2000), 747-756.
		\medskip
		
		\bibitem{Noor-resolvent}
		M.A. Noor, \textit{Resolvent Dynamical Systems for Mixed Vriational Inequalities}, Korean J. Comput. \& Appl. Math. 9 1(2002), 15-26.
		\medskip
		
		\bibitem{Noor-WH}
		M.A. Noor, \textit{Some developments in general variational inequalities}, Applied Mathematics and Computation. 152 (1)(2004,) 199-277.
		\medskip  
		
		
		\bibitem{Noor-inequality}
		M.A. Noor, K.I. Noor, R. Latif, \textit{Dynamical Systems and Variational Inequalities}, Journal of Inequalities and Special Functions, 8, 11 (2017), 22-29.
		
		
		
		\bibitem{Noor-Oettli}
		M.A. Noor, W. Oettli, \textit{On general nonlinear complementarity problems and quasi-equilibria}, Le Matematiche (Catania), 49(1994), 313—331.
		\medskip 
		
		\bibitem{Stampacchia}
		G. Stampacchia, \textit{Formes bilineaires coercitives sur les ensembles convexes}, Comptes Rendus Acad. Sci. Paris, 258(1964), 4413-4416.
		
		
		
	\end{thebibliography}
	
\end{document}